\documentclass[10pt]{article}%
\usepackage{graphicx}
\usepackage{pdfsync}
\usepackage{amssymb,amsthm,amsmath}

\usepackage{color}

\title{Homogenization of ferromagnetic energies\\ on Poisson random sets  in the plane}
\author{
{\sc Andrea Braides
}
\\ Dipartimento di Matematica,
 Universit\`a di Roma `Tor Vergata'\\
via della Ricerca Scientifica, 00133 Rome, Italy\\
\and
{\sc Andrey Piatnitski}
\\
The Arctic University of Norway, UiT,  Campus
Narvik,\\ P.O. Box 385, Narvik 8505, Norway \\ and \\ Institute for Information Transmission Problems
of RAS,\\ 127051 Moscow, Russia 
\\
}
\date{
}                                      

\newtheorem{definition}{Definition}[section]
\newtheorem{lemma}[definition]{Lemma}
\newtheorem{theorem}[definition]{Theorem}
\newtheorem{proposition}[definition]{Proposition}

\newtheorem{remark}[definition]{Remark}

\def\e{\varepsilon}

\def\ZZ{\mathbb{Z}}
\def\NN{\mathbb{N}}
\def\N{{\mathcal N}}\def\E{{\mathcal E}}\def\V{{\mathcal V}}
\def\rr{\mathbb{R}}
\def\HH{\mathcal{H}}
\def\I{{\mathcal I}}
\def\wto{\rightharpoonup}
\def \trait (#1) (#2) (#3){\vrule width #1pt height #2pt depth #3pt}
\def \qed{\hfill
        \trait (0.1) (6) (0)
        \trait (6) (0.1) (0)
        \kern-6pt
        \trait (6) (6) (-5.9)
        \trait (0.1) (6) (0)
\medskip}

\begin{document}
\maketitle

\noindent
{\bf Abstract.} We prove that by scaling nearest-neighbour ferromagnetic energies defined on Poisson random sets in the plane
we obtain an isotropic perimeter energy with a surface tension characterised by an asymptotic formula. The result relies on proving that cells with `very long' or `very short' edges of the corresponding Voronoi tessellation can be neglected. In this way we may apply Geometry Measure Theory tools to define a compact convergence, and a characterisation of metric properties of clusters of Voronoi cells using limit theorems for subadditive processes.
\bigskip


\section{Introduction}
In this paper we study a prototypical model of pair-interaction energies on Poisson random sets in the plane. These energies are a random extension of nearest-neighbour `ferromagnetic' systems defined on Bravais lattices, whose overall behaviour is that of an interfacial energy \cite{CDL,ABC}. The possibility of lattice approximations for arbitrary interfacial energies 
makes the analysis of ferromagnetic energies relevant for numerical approximations and modeling issues (we refer to \cite{BK} for optimal constructions on regular lattices, available even with constraints on the interaction potentials). Surface energies are an important building block in the study of general functionals defined on more complex spaces of functions of bounded variation, passing through the generalization to functions with a discrete number of values \cite{AB} and using the latter to approximate arbitrary functions by coarea-type arguments (see e.g.~\cite{A}). Furthermore, the study of energies involving bulk and surface part can often be decoupled in the analysis of each part, which justifies the analysis of surface energies  separately also in that context (see \cite{BDV,DZ} and the recent advances in the analysis and derivation of complex energies from discrete systems in \cite{BBC,BBZ}). The present contribution can be then viewed as a step towards the extension of the analysis of discrete systems producing bulk and surface integrals to general random distribution of points. The simplest case of parameters taking only two values (equivalently, characteristic functions) will allow us to concentrate on the basic geometric features of the underlying discrete environment.

Discrete energies with randomness producing surface effects have been previously considered under various hypotheses. Results on regular lattices with random interactions comprise: random weak membrane models in \cite{BP-ARMA}, random ferromagnetic energies with positive coefficients in \cite{BP-JFA} and
ferromagnetic energies with a random distribution of degenerate coefficients in \cite{BP-JOSS}. Stochastic lattices have been considered under the hypothesis that sites be distributed in such a way that no `great holes' of `concentration of sites' may occur, so that we obtain uniform upper and lower estimates for the size of the Voronoi cells of the underlying tessellation, which implies that these lattices can be treated in average as a regular periodic lattice (see \cite{BLB,ACG,ACR,BCR}). Our focus is precisely in avoiding such an hypothesis considering points distributed according to a Poisson point process in the plane ({\em Poisson random set}\,). We denote by $\N$ such a set of points and by $\E$ the set of the edges of the underlying Delaunay triangulation, which are identified with pairs of points $(i,j)$ in $\N\times \N$ (the {\em nearest neighbours} in $\N$). The energy we consider can be viewed as defined on subsets $\mathcal I$ of $\N$ by
\begin{equation}
E({\mathcal I})=\#\{(i,j)\in\E: i\in{\mathcal I}, j\not\in{\mathcal I}\}.
\end{equation}
Note that the same energy can be interpreted as the number of edges of the boundary of the set
\begin{equation}
A_{\mathcal I}=\bigcup_{i\in{\mathcal I}} C_i,
\end{equation}
where $C_i$ is the cell of the Voronoi tessellation containing the point $i\in\N$. Another way to write the same energy
is by identifying each set ${\mathcal I}$ with a (scalar) spin function parameterized by indices in $\N$ and defined by
$u^\I_i=1$ if $i\in \I$ and $u^\I_i=-1$ if $i\not\in \I$, so that we may rewrite $E({\mathcal I})$ as depending on $u^\I$, setting
\begin{equation}
E(u^\I)={1\over 8}\sum_{i,j\in\N}(u_i-u_j)^2={1\over 4}\sum_{i,j\in\N}|u_i-u_j|,
\end{equation}
the factors coming from double counting and the fact that $|u_i-u_j|\in\{0,2\}$. Conversely, we may take this as the definition of the energy, and correspondingly pass to subsets of $\N$ by noting that $E(\I_u)=E(u)$, where $\I_u=\{I\in\N: u_i=1\}$.

In order to describe the overall properties of $E$ we perform a discrete-to-continuum analysis through a scaling procedure.
We intoduce a small parameter $\e>0$ and consider the scaled energy $E_\e$ defined on subsets of $\e\E$ by
\begin{equation}
E_\e({\mathcal I})=\e\,\#\{(i,j)\in\e\E: i\in{\mathcal I}, j\not\in{\mathcal I}\},
\end{equation}
which again can be interpreted as $\e$ times the number of edges of the boundary of the scaled set
\begin{equation}
A^\e_{\mathcal I}=\e \bigcup_{i\in{\mathcal I}} C_{i/\e}.
\end{equation}

Note that if we had a uniform upper and lower bound of the size of each of these edges,
then $E_\e({\mathcal I})$ would be comparable with the perimeter of $A^\e_{\mathcal I}$.
In that case, given a family ${\mathcal I}_\e$ with equibounded $E_\e({\mathcal I})$,
the sets $A_\e=A^\e_{\mathcal I_\e}$ would be (locally) precompact in the sense of sets of finite perimeter; i.e.,
there would exist a set of finite perimeter $A$ such that, up to subsequences,
$|(A_\e\triangle A) \cap Q|\to 0$ for any cube $Q$.

For Poisson random sets, the edges of Voronoi cells do not satisfy a uniform estimate. Nevertheless,
very long or very short edges are in a sense negligible. Indeed, a result by Pimentel \cite{Pi}
implies that a path in which a large proportion of such sets is present must be `short', and hence,
by an isoperimetric argument encircle a `small' set. Using this result, we can show that if
$E_\e({\mathcal I})$ is equibounded and $A_\e$ are defined above, then there exists families
of sets $B'_\e$ and $B''_\e$ such that $|B'_\e\cup B''_\e|\to 0$ and the perimeter of the sets
\begin{equation}
(A_\e\cup B'_\e)\setminus B''_\e
\end{equation}
is equibounded. We deduce then that, up to subsequences, $A_\e$ still (locally) converge
to a set of finite perimeter $A$.

We can then characterize the behaviour of the energies $E_\e$ by computing their $\Gamma$-limit with respect to this convergence. Note that, by the isotropy of Poisson random sets, if the limit is of perimeter type, it must be of the form
\begin{equation}
F(A)=\tau_0\,\HH^1(\partial A);
\end{equation}
i.e., a constant $\tau_0$ (the {\em surface tension}) times the perimeter of $A$ (in this notation $\partial A$ denotes the
{\em reduced boundary} of $A$). The main issue is then to characterize such $\tau_0$ so as to adapt the discrete-to-continuum technique of \cite{BP-JFA,BP-JOSS} to this case. A central observation is that the union of the boundaries of all Voronoi cells $C_i$ for which we have suitable outer and inner bounds determine a set which possesses a unique infinite connected component. We then introduce a parameter $\alpha>0$ that quantifies these bounds so that they become less and less stringent when $\alpha\to 0$. We denote by $\V_\alpha$ the union of boundaries of such Voronoi cells.
The properties of $\V_\alpha$ are derived from percolation argument as in \cite{BP-JOSS,BP-JFA,BP-ARMA}, and can be used to prove that a first-passage percolation formula holds for paths in $\V$ and at the same time permit to use the blow-up technique \cite{FM,BMS} for proving a lower bound.
An upper bound is finally shown by using the subadditive properties of the problems defining $\tau_0$.

It is worth noting that some of the results extend to arbitrary dimension (mainly, the compactness lemma for sets with equibounded energy), but the properties of regular Voronoi cells as stated and the characterization of $\tau_0$ with a first-passage percolation formula are particular of the planar case. The treatment of the asymptotic analysis of the energies in higher dimension will require different tools and homogenization formulas, which justify a separate treatment.

\section{Notation and statement of the results}

${\cal L}^2(A)$ or $|A|$ denotes the $2$-dimensional Lebesgue measure of a set $A$, ${\bf 1}_A$ the characteristic function of the set $A$,
$Q=[-1/2,1/2]^2$ the unit cube in $\rr^2$.

\subsection{Poisson random sets}

$\N$ denotes a {\em Poisson random set} with {\em intensity $\lambda>0$} in $\mathbb R^2$ defined on a probability space $(\Omega,\mathcal{F},\mathbf{P})$. We refer for instance to \cite{DV_J} for the definition of a Poisson random set and its main properties.  We recall that $\N$ is almost surely a locally finite subset of $\mathbb R^2$ such that\\
$\bullet$ for any bounded Borel set $B\subset\mathbb R^2$ the number of points in $B\cap \N$ has a Poisson law with parameter $\lambda|B|$
$$
\mathbf{P}\{\#(B\cap\N)=n\}=e^{-\lambda|B|}\frac{(\lambda|B|)^n}{n!};
$$
$\bullet$ for any collection of bounded disjoint Borel subsets in $\mathbb R^2$ the random variables defined as the number
of points of $\N$  in these subsets are independent.

We also assume that the probability space is equipped with a dynamical system $T_x:\Omega\mapsto\Omega$, $x\in\mathbb R^2$,
and that for any bounded Borel set $B$ and any $x\in\mathbb  R^2$ we have  $\#\big((B+x)\cap\N\big)(\omega)
=\#\big(B\cap\N\big)(T_x\omega)$.  We recall that $T_x$ is a group of measurable measure preserving transformations in $\Omega$,
also measurable as a function $T_\cdot\,:\,\Omega\times\mathbb R^2\mapsto\Omega$. We  suppose that $T_x$ is ergodic. 
Further details can be found for instance in \cite[Chapter 7]{JKO}.

 In what follows, we only consider a Poisson random set with intensity $1$, since the results for as Poisson random set with intensity $\lambda$ may be obtained by considering the case with intensity $1$ and then apply a scaling transformation $\N\,\longrightarrow \sqrt{\lambda}\N$.

The cells of the {\em Voronoi tessellation} of $\N$ are denoted by
$$C_i:=\{x\in\rr^2:|x-i|\le |x-j| \hbox{ for all } j\in\N\}
$$
 with $i\in\N$. Each $C_i$ thus defined is a polyhedral set; the set of edges of the Voronoi cells is denoted by $\cal V$. The set of the vertices of $C_i$ (or endpoints of elements of $\cal V$) is denoted by $\N^*$

Note that we may assume that each point in $\rr^2$ belongs to at most three Voronoi cells or three elements of $\E$, since this is an event of probability $1$.

The set of edges of the {\em Delaunay triangulation} of $\N$ is denoted by $\E$ and is identified with the set of pairs $(i,j)$ in $\N\times\N$ such that $C_i$ and $C_j$ share a common edge.

We define a {\em path} of Voronoi cells as a collection $\{C_{i_j}: 1\le j\le K\}$ such that $C_{i_j}$ and $C_{i_{j+1}}$
have an edge in common, or, equivalently, such that $(i_j, i_{j+1})\in\E$ for all $j\in\{1,\ldots, K-1\}$.
From the latter standpoint, we also talk of a path in $\E$.
We say that such a path connects two sets $X$ and $Y$ if $X\cap C_1\neq \emptyset$ and $Y\cap C_K\neq \emptyset$.
If $X=\{x\}$ and $Y=\{y\}$ then we simply say that the path connects $x$ and $y$.

\subsection{Asymptotic behaviour of ferromagnetic energies on Poisson random sets}
For future reference and comparison with the existing literature, we state our results in terms of energies on (scalar) spin functions,
keeping in mind the possible alternative formulations as energies on sets or on set of points.
The (scaled) ferromagnetic energy of the Poisson random set is defined on spin functions $u:\e\N\to\{-1,1\}$ by
\begin{eqnarray}\label{energies}
\nonumber
E_\e(u)&=&{1\over 8}\sum_{(i,j)\in\e\E}\e (u_i-u_j)^2\\&=&\nonumber
{1\over 2}\e\, \#\{(i,j)\in \e\E: u_i\neq u_j\}
\\ \ \nonumber\\
&=&\e\, \#\{(i,j)\in \e\E: u_i=1, u_j=-1\},
\end{eqnarray}
where the scaling factor $1\over 8$ is due to double counting and to the fact that $(u_j-u_j)^2\in\{0,4\}$.

To each $u:\e\N\to\{-1,1\}$ we associate the (scaled) {\em Voronoi set of $u$} defined by
\begin{equation}
V_\e(u)=\bigcup_{\{i: u_i=1\}} \e\, C_{i/\e},
\end{equation}
and the {\em piecewise-constant interpolation} (with underlying set $\e\N$), still denoted $u:\rr^2\to\{-1,1\}$, defined by
\begin{equation}
u(x)=\begin{cases} 1 & \hbox{if } x\in V_\e(u)\\
-1 & \hbox{if } x\in\rr^2\setminus V_\e(u).
\end{cases}
\end{equation}

\begin{definition}\label{conv-L1}
We say that a family $u^\e:\e\N\to\{-1,1\}$ {\em converges to a set $A$} if the piecewise-constant interpolations $u^\e$ converge to the function ${\bf 1}_A-{\bf 1}_{\rr^2\setminus A}$ locally in $L^1(\rr^2)$, or, equivalently, if ${\bf 1}_{V_\e(u^\e)}$ converge to ${\bf 1}_A$  locally in $L^1(\rr^2)$.
\end{definition}

The following compactness lemma justifies the use of the convergence in Definition \ref{conv-L1} in the computation of the $\Gamma$-limit of $E_\e$ \cite{GCB,Handbook}. Note that the result cannot be directly deduced from the compactness property of sets of equibounded perimeter, since we cannot deduce the equiboundedness of the perimeters of the sets $V_\e(u^\e)$ from the equiboundedness of $E_\e(u_\e)$.

\begin{lemma}[compactness]\label{set-compactness}
Let $u^\e$ be such that $\sup_\e E_\e(u^\e)<+\infty$. Then, up to subsequences $u^\e$ converges to some set $A$ in the sense of Definition {\rm\ref{conv-L1}}. Moreover, the limit set is a set of finite perimeter.
\end{lemma}


The compactness lemma above ensures that the domain of the $\Gamma$-limit of $E_\e$ be the family sets of finite perimeter in $\rr^2$.
The asymptotic behaviour of $E_\e$ will be described by an asymptotic formula similar to those encountered in first-passage percolation, involving minimal paths on $\E$ between points of $\rr^2$. To that end we define for all $x\in\rr^2$
$$
\pi_0(x)=\hbox{ closest point of $\N^*$ to $x$}.
$$
For almost all $x$ this point is uniquely defined. For the remaining points we choose one of the closest points of $\N^*$ to $x$.
%
%
For $x,\,y\in\mathbb R^2$ we define
\begin{equation}\label{em0}
m_0(x,y)=\min\{\#\{e_k\}: \{e_k\} \hbox{ is a path in $\mathcal{E}$ connecting $\pi_0(x)$ and $\pi_0(y)$}\},
\end{equation}
where a path of segments (in our case edges in $\cal V$) connecting two points $\overline x$ and $\overline y$ is a collection of segments $[x_{k-1}, x_k]$ with $1\le k\le K$ for some $K\in\NN$ such that $x_0=\overline x$ and $x_K=\overline y$, and such that the related piecewise-linear curve is not self-intersecting.

\begin{theorem}[homogenization theorem]\label{HoT}
 Let $\cal E$ be a Poisson random set with intensity $1$. Then there exists a deterministic constant $\tau_0\in (0,+\infty)$ (the {\em surface tension}) such that almost surely the energies $E_\e$ defined in \eqref{energies} $\Gamma$-converge to the energy $F(A)= \tau_0\,{\cal H}^1(\partial A)$, defined on sets of finite perimeter, with respect to the convergence in Definition {\rm\ref{conv}}. Furthermore the constant $\tau_0$ satisfies
$$
\tau_0=\lim\limits_{t\to\infty}\frac{m_0\big((0,0),(t,0)\big)}t
$$
almost surely, where $m_0$ is given by \eqref{em0}.
\end{theorem}



The proof of this result will be the content of Section \ref{s_peop_st}.




\begin{remark}\rm By the scaling argument ${\cal N}\to \sqrt\lambda\cal N$, we deduce that if $\cal E$ is a Poisson random set with intensity $\lambda$ then  the $\Gamma$-limit of the corresponding $E_\e$ is $\sqrt\lambda\,\tau_0\,{\cal H}^1(\partial A)$.
\end{remark}

\section{Compactness}
This section is devoted to the proof of the Compactness Lemma \ref{set-compactness}. Even though we will use it in the planar case, we note that that result can be proved in any space dimension $d$ up to minor changes (see Remark \ref{record} below).

$\Pi$ denotes the set of finite connected unions of Voronoi cells (here connected means that the corresponding set of edges of the Delaunay triangulation is connected).
 If $P\in\Pi$ we set
$$
{\bf A}(P)=\{ z\in \ZZ^2: (z+Q)\cap P\neq\emptyset\}.
$$

\begin{lemma}[Pimentel's polyomino lemma \cite{Pi}]\label{pimentel}
Let $R>0$ and $\gamma>0$. Then there exists a deterministic constant $C$ such that for almost all $\omega$ there exists $\e_0=\e_0(\omega)>0$ such that
if $P\in\Pi$ and $\e<\e_0$ satisfy
\begin{equation}
P\cap {R\over\e}Q\neq\emptyset,\qquad
\max\Bigl\{\# \{i: C_i\subset P\},\# {\bf A}(P)\Bigr\}\ge \e^{-\gamma}
\end{equation}
then we have
 \begin{equation}\label{eqpap}
{1\over C} \# \{i: C_i\subset P\}\le \# {\bf A}(P)\le C\,\# \{i: C_i\subset P\}.
\end{equation}
%
\end{lemma}

Note in particular that in the hypotheses of the lemma, we also have
\begin{equation}
\min\Bigl\{\# \{i: C_i\subset P\},\# {\bf A}(P)\Bigr\}\ge {1\over C}\e^{-\gamma}\,.
\end{equation}
Further geometric properties of such Voronoi tessellations can be found in \cite{Calka}.

Lemma \ref{set-compactness} will be a consequence of the following result.

\begin{lemma}[compactness of Voronoi sets]\label{Voronoi-comp}
Let $u^\e$ be such that $\sup_\e E_\e(u^\e)<+\infty$. Then we can write
$$
V_\e(u^\e)= (A_\e\cup B'_\e)\setminus B''_\e,
$$
where $|B'_\e|+|B''_\e|\to 0$,  the family ${\bf 1}_{A_\e}$ is precompact
in $L^1_{\rm loc}(\rr^2)$ and each its limit point is the characteristic
function of a set of finite perimeter $A$, so that the same  holds for ${\bf 1}_{V_\e(u^\e)}$.
\end{lemma}

\begin{proof} Since we reason locally, in order to ease the notation we
assume that e.g.~all $u^\e$ are identically $-1$ outside a fixed cube
(or equivalently that $V_\e(u^\e)$ are contained in a fixed cube).

We fix $\gamma>0$ small enough. We subdivide $\partial V_\e(u^\e)$ into its connected components.
We denote by ${\mathcal C}^{\gamma,+}_\e$ the family of such connected components $S$ with
\begin{equation}\label{condc+}
\#\{ i\in \N\,: u_{\e i}=1, \e C_{i}\cap S\neq\emptyset \}\geq \e^{-\gamma}.
\end{equation}
Note that each such connected component can be identified with the set
\begin{equation}\label{def-P}
P=P(S)=\bigcup \Bigl\{ C_{i}: u_{\e i}=1, \e C_{i}\cap S\neq\emptyset\Bigr\},
\end{equation}
which belongs to the set $\Pi$. We denote by ${\mathcal C}^{\gamma,-}_\e$ the family of the remaining connected components.

The first step will be to identify the small sets $B'_\e$ and $B''_\e$ as the `interior' of contours in
${\mathcal C}^{\gamma,-}_\e$ where the inner trace of ${\bf 1}_{V_\e(u^\e)}$ is $0$ and $1$, respectively.
In this way the remaining set will have a boundary only composed of `large' components from ${\mathcal C}^{\gamma,+}_\e$. This argument needs a little more formalization since we may have contours contained in other contours.

By the finiteness of the energy we have
$$
\#{\mathcal C}^{\gamma,-}_\e\le {C\over\e}
$$
Note that
$$
\#\Bigl({\bf A}\Bigl({1\over\e}S\Bigr)\Bigr)\le C\e^{-\gamma}\hbox{ for every }S\in {\mathcal C}^{\gamma,-}_\e.
$$
 Indeed, otherwise $\#({\bf A}({1\over\e}S))> C\e^{-\gamma}>\e^{-\gamma}$, so that the hypotheses of Lemma \ref{Voronoi-comp} are satisfied and \eqref{eqpap} implies that \eqref{condc+} holds, which gives a contradiction. Hence each $S\in {\mathcal C}^{\gamma,-}_\e$ is contained in a set with boundary at most of length $C\e^{1-\gamma}$.
By an isoperimetric estimate, the measure of the bounded set sorrounded by each $S\in {\mathcal C}^{\gamma,-}_\e$ is  $O(\e^{2-2\gamma})$. Hence, the total measure of such sets is $O(\e^{1-2\gamma})$.

Consider now each maximal $S\in {\mathcal C}^{\gamma,-}_\e$; i.e., which is not contained in any other bounded set whose boundary is another element in ${\mathcal C}^{\gamma,-}_\e$. For each such $S$, let $P$ be defined from $S$ by \eqref{def-P}. We have two cases, whether $\e P$ is interior to $S$ or not. We denote by ${\mathcal C}^{\gamma,-}_{1,\e}$ the first family, by ${\mathcal C}^{\gamma,-}_{2,\e}$ the second one, and define
$B'_\e$ as the union of the $\e C_{i/\e}$ in the interior of $S$ for some $S\in {\mathcal C}^{\gamma,-}_{1,\e}$ and such that $u^\e_i=1$, and $B''_\e$ as the union of the $\e C_{i/\e}$ in the interior of $S$ for some $S\in {\mathcal C}^{\gamma,-}_{2,\e}$ and such that $u^\e_i=-1$. If we set
$$
A_\e=(V_\e(u^\e)\setminus B'_\e)\cup B''_\e
$$
then $\partial A_\e$ consists only of components in ${\mathcal C}^{\gamma,+}_\e$,
and
$$
|B'_\e\cup B''_\e|\le C\e^{1-2\gamma}.
$$

In order to prove the compactness of $A_\e$ we  write
$A_\e= A'_\e\cup A''_\e$, where
\begin{eqnarray*}
	A'_\e&=&\bigcup\{(\e z+\e Q): \e z+\e Q)\subset \partial A_\e\}\\
	A''_\e&=&A_\e\setminus A'_\e
\end{eqnarray*}
Note that
$$
\partial A'_\e\subset \e\bigcup_{S\in {\mathcal C}^{\gamma,+}_\e}\partial{\bf A}(P(S))
$$
with $P(S)$ defined in \eqref{def-P}.
By Lemma \ref{pimentel} we have
$$
{\mathcal H}^1(\partial {\bf A}(P(S)))\le C\#\{ i\in \N\,: u_{\e i}=1, \e C_{i}\cap S\neq\emptyset \}
$$
Summing up over all $S\in  {\mathcal C}^{\gamma,+}_\e$ we obtain
$$
{\mathcal H}^1(\partial A'_\e)\le C\,E_\e(u^\e).
$$
Hence, the functions ${\bf 1}_{A'_\e}$ are locally precompact in $L^1(\rr^2)$ by the precompactness of sets of equibounded perimeter \cite{LN98,Maggi}.
%

Again by Lemma \ref{pimentel} we have
$$
|A''_\e|\le C\e^2\sum_{S\in {\mathcal C}^{\gamma,+}_\e}\#{\bf A}(P(S))\le C\e E_\e(u^\e).
$$
This shows that $|A''_\e|\to 0$, and proves the claim.
\end{proof}

\begin{remark}\label{record}\rm
The previous compactness result holds in any dimension $d$ with minor changes in the proof, upon noting that Pimentel's lemma holds with
$$
{\bf A}(P)=\{ z\in \ZZ^d: (z+Q)\cap P\neq\emptyset\}
$$
and $Q$ the coordinate unit cube in $\rr^d$  \cite{Pi}.
\end{remark}

\begin{remark}[convergence in terms of the empirical measures]\label{conv}\rm
To each $u^\e:\e\N\to\{-1,1\}$ we can associate the so-called {\em empirical measure}
$$
\mu(u^\e)=\sum_{\{i\in\e \N : u^\e_i=1\}}\e^2 \delta_i.
$$
If $u^\e$ are such that $\sup_\e E_\e(u^\e)<+\infty$ and $u^\e$ converge to $A$ as in Definition \ref{conv-L1},
then the measures $\mu(u^\e)$ locally converge to the measure ${\bf 1}_A{\mathcal L}^2$ with respect to the weak$^*$ convergence of measures. Thanks to Lemma \ref{Voronoi-comp}, then these two convergences are equivalent.

To check the convergence of $\mu(u^\e)$, we first note that we may suppose that
$
\mu(u^\e)\wto f{\mathcal L}^2
$
for some $f:\rr^2\to[0,1]$. It suffices to show that $f=0$ at almost every point of density $0$ for $A$ (a symmetric argument then shows that $f=1$ at almost every point of density $1$ for $A$).

For almost all such $x_0$ we have that $$\limsup\limits_{\e\to 0}|V_\e(u^\e)\cap (x_0+\rho Q)|= o(\rho^2)$$ and
$\limsup\limits_{\e\to 0}E_\e(u^\e,Q_\rho)= o(\rho)$, where we have set
\begin{eqnarray}\nonumber
E_\e(u^\e,Q_\rho)={1\over 2}\e\, \#\{(i,j)\in \e\E: u^\e_i\neq u^\e_j, i\hbox{ or }j\in \rho Q\}.
\end{eqnarray}

We may subdivide  $V_\e(u^\e)\cap (x_0+\rho Q)$ into disjoint connected components:
$$
V_\e(u^\e)\cap (x_0+\rho Q)= \bigcup_{\# (P_j\cap \e\N)\le \e^{-\gamma}}P_j
\cup \bigcup_{\# (L_k\cap\e\N)> \e^{-\gamma}}L_k,
$$

We may apply Lemma \ref{pimentel} to each $L_k$ to obtain
$$
\sum_k\e^2\#(L_k\cap \e\N)\le C\e^2\sum_k\#{\bf A}\Bigl({1\over \e}L_k\Bigr)\le C|V_\e(u^\e)|= o(\rho^2).
$$
As for $P_j$ we have
$$
\#(\{P_j\})\le  {1\over\e}E_\e(u^\e,Q_\rho)= {1\over\e} o(\rho),\qquad
\sum_j\#(P_j\cap \e\N)\le {1\over \e^{1+\gamma}} o(\rho).
$$

In conclusion,
$$
\mu(u^\e)(x_0+\rho Q)=
\e^2\#\{u^\e_i=1, \ i\in x_0+\rho Q\}\le o(\rho^2)+\e^{1-\gamma}o(\rho).
$$
Letting first $\e\to0$ and then $\rho\to0$ we prove the claim.
\end{remark}

\section{Proof of the Homogenization Theorem}
\label{s_peop_st}

In this section we prove Theorem \ref{HoT}, first characterizing the surface tension and then computing the $\Gamma$-limit.
Preliminarily, we introduce  regular Voronoi cells and study their geometry.

\subsection{Geometry of clusters of regular Voronoi cells}
The surface tension characterizing the $\Gamma$-limit will be expressed by an asymptotic average length of minimal paths analogous to first-passage percolation formulas. A difficulty in our case is that in principle one of the end-points of such paths could be located in an `exceptional region' where very small Voronoi cells accumulate. In order to treat this case, we first introduce regular Voronoi cells and study some percolation characteristics of the grid of such cells.

For $\alpha>0$ we set
\begin{equation}
\N^0_\alpha=\Bigl\{ i\in \N: C_i\hbox{ contains a ball of radius }\alpha\hbox{, diam}\,C_i\le {1\over\alpha}, \#\hbox{edges of }C_i\le {1\over\alpha}\Bigr\}
\end{equation}
the family of {\em regular Voronoi cells with parameter $\alpha$}.
The following lemma describes some geometrical features of regular Voronoi tessellations.

\begin{lemma}[a channel property of $\N^0_\alpha$]\label{chalpha} 
Let $\delta>0$. For every $T\in\rr$, $\nu\in S^1$ and $x\in\rr^2$ we define
$$
R^\nu_{T,\delta}(x)=\Bigl\{x: |\langle x-x_i,\nu_i\rangle|\le \delta T, |\langle x-x_i,\nu_i^\perp\rangle|\le {1\over 2}T\Bigr\}.
$$
Then there exist $\alpha_0, C_\delta>0$ such that a.s.~there exists $T_0(\omega)>0$ such that for all $T>T_0(\omega)$
the rectangle $R^\nu_{T,\delta}(x)$ contains at least $C_\delta T$ disjoint paths of Voronoi cells
$C_i$ with $i\in\N^0_\alpha$ connecting the two opposite sides of $R^\nu_{T,\delta}(x)$ parallel to $\nu$.
This property is uniform as $x/T$ vary on a bounded set of $\rr^2$.
\end{lemma}

\begin{proof}  Our arguments rely on the result known as channel property in the Bernoulli site percolation model in $\mathbb Z^2$.
Denote $Q_{5L}:=[-5L,5L]^2$,
and for $L,\,K,\,\alpha\in\mathbb R^+$ and $j\in\mathbb Z^2$ denote by $\mathcal{E}(L,K, \alpha,j)$ the event that the following
conditions are fulfilled:
\begin{itemize}
\item[$(\mathbf{c}_1)$] any square $[0,L]^2+Li$ with $i\in\mathbb Z^2\cap[-4.5,5.5]^2$ contains at least one point of $\mathcal{N}-10j$,
\item[$(\mathbf{c}_2)$] the total number of points  $\#((\mathcal{N}-10Lj)\cap Q_{5L})$ does not exceed $K$,
\item[$(\mathbf{c}_3)$] the distance between any two points of $(\mathcal{N}-10Lj)\cap Q_{5L}$ as well as the distance from any point of
$(\mathcal{N}-10Lj)\cap Q_{5L}$ to $\partial Q_{5L}$ is greater than $2\alpha$.
\end{itemize}
Letting $\xi_j$ be the characteristic function of $\mathcal{E}(L,K, \alpha,j)$ and considering the properties of the Poisson random set
we conclude that $\xi_j$, $j\in \mathbb Z^2$,
are i.i.d. random variables.
For any $\gamma>0$ one can choose sufficiently large $L$ and $K$ and sufficiently small $\alpha>0$
so that
\begin{equation}\label{prob_cal_e}
\mathbf{P}(\mathcal{E}(L,K, \alpha,j))>1-\gamma.
\end{equation}
Indeed, the probability that any cube of size $L$ in $Q_{5L}$
contains at least one point of the Poisson random set tends to $1$ as $L\to\infty$.  Then, given $L>0$, the probability that the number of points in $Q_{5L}$
does not exceed $K$ tends to $1$ as $K\to\infty$. The probability that in the cube $Q_{5L}$ the smallest distance between two points is less than $\alpha$ goes to zero as $\alpha\to0$. Finally, the probability that $\alpha$-neighbourhood of $\partial Q_{5L}$
contains at least one point also goes to zero. Combining this relations we obtain the desired property.

For any two points $j'\,,\,j''\in\mathbb Z^2$ such that $|j'-j''|=1$ denote by $I_{j',j''}$ the segment $[10Lj',10Lj'']$ in $\mathbb R^2$.
If $\xi_{j'}=\xi_{j''}=1$ then
\begin{itemize}
  \item[(\textbf{s}1)] any Voronoi cell $C_i$ that has a non-trivial intersection with $I_{j',j''}$ belongs to $(Q_{5L}+10Lj')\cup(Q_{5L}+10Lj'')$,
  \item[(\textbf{s}2)]  any such a cell $C_i$ contains a ball of radius $\alpha$,
  \item[(\textbf{s}3)] the number of edges of each such $C_i$ is not greater than $K$.
\end{itemize}
In particular, due to (\textbf{s}1) and $(\mathbf{c}_2)$, the total number of the cells $C_i$ having a non-empty intersection with $I_{j',j''}$ does not exceed $2K$.

Statement (\textbf{s}1) can be justified as follows: Let $x'$ be an arbitrary point of $I_{j',j''}$. Denote by $C_i$ the Voronoi cell
that contains $x'$ and by $x_i$ the corresponding point of the Poisson random set. Due to $(\mathbf{c}_1)$ we have
$|x'-x_i|\leq \sqrt{2}L$.    Then any point $y\in\partial\big((Q_{5L}+10Lj')\cup(Q_{5L}+10Lj'')\big)$ satisfies the inequality
$|y-x_i|\geq (5-\sqrt{2})L$.  On the other hand, by $(\mathbf{c}_1)$ the distance of $y$ from $\N$ is not greater than $\sqrt{2}L$. This implies that
$y\not\in C_i$. Therefore, $C_i\subset (Q_{5L}+10Lj')\cup(Q_{5L}+10Lj'')$, and  (\textbf{s}1) follows.

In a similar way one can show that for any $C_i$ that has a nontrivial intersection with $I_{j',j''}$ and any $x_j\in \mathcal{N}$ such that $C_i$ and $C_j$ have an edge in common we have
$x_j\in(Q_{5L}+10Lj')\cup(Q_{5L}+10Lj'')$. In view of  $(\mathbf{c}_2)$ this yields (\textbf{s}3).

Statement (\textbf{s}2) is an immediate consequence of $(\mathbf{c}_3)$.

Now the desired channel property follows from the well-known channel property in the Bernoulli site percolation model.
For the reader convenience we formulate it here.  Let $\eta_j$, $j\in\mathbb Z^2$, be a collection of i.i.d. random variables
taking on the value $1$ with probability $p$ and the value $0$ with probability $(1-p)$. We say that $\{j_i\}_{i=1}^M$ is a
$1$-path if $j_i$ and $j_{i+1}$, $i=1,\,2,\ldots,M-1$,
 are neighbouring points of $\ZZ^2$ and $\eta_{j_i}=1$ for all $i$. Then there exists $p_{\rm cr}\in(0,1)$
such that for all $p>p_{\rm cr}$ the following statement holds:
for any $\delta>0$   there exists $\mathcal{K}_\delta>0$ such that for almost each $\omega\in\Omega$ there exists
$T_0=T_0(\omega)>0$ such that  any rectangle  $R^\nu_{T,\delta}(x)$
with $T\geq T_0$ and $x\in[-T,T]^2$ contains at least $\mathcal{K}_\delta>0$ disjoint $1$-paths
connecting the two opposite sides of $R^\nu_{T,\delta}(x)$ parallel to $\nu$.
We refer to \cite{Kes} for further details.

It remains to choose $\gamma$  in \eqref{prob_cal_e} in such a way that $1-\gamma>p_{\rm cr}$. Labeling
the squares $Q_{5L}+10j$ with the corresponding points $j\in \ZZ^2$ and recalling the just formulated channel property
of the Bernoulli site percolation model with $\eta_j=\xi_j$ we obtain the desired statement.
\end{proof}

From the proof of the previous lemma, in particular we obtain the following proposition.

\begin{proposition}\label{conn-alfa}
There exists $\alpha_0$ such that if $\alpha<\alpha_0$ there exists a unique infinite connected component of $\N^0_\alpha$, and its complement is composed of bounded connected sets. \end{proposition}

With this proposition in mind, we may define {\em clusters} of regular Voronoi cells.

\begin{definition}[$\alpha$-clusters]\label{clusa}
Let $\alpha<\alpha_0$ be as in Proposition {\rm\ref{conn-alfa}}. We denote by $\N_\alpha$ the infinite connected component of $\N^0_\alpha$ defined therein. Moreover, we denote by $\N_\alpha^*$ the set of vertices of edges of $C_i$ with $i\in \N_\alpha$,
by $\V_\alpha$ the set of the edges of such $C_i$, and by $\E_\alpha$ the set of edges of the Delaunay triangulation defined by set of pairs $(i,j)$ in $\N_\alpha^2$ such that $C_i$ and $C_j$ share a common edge.
\end{definition}

\begin{remark}[a channel property of $\N_\alpha$]\rm
With the notation of Definition \ref{clusa}, note that the paths of cells $C_i$ in Lemma \ref{chalpha} can be taken with $i\in \N_\alpha$.
\end{remark}

\subsection{Geometric properties of Voronoi tessellation of Poisson set. Surface tension }
\label{ss_surf_te}
In this section we consider the geometric properties of the Poisson-Voronoi tessellation
and introduce the surface tension in terms of an asymptotic distance between two  points of the grid.
In order to apply the subadditive theorem we should show that the grid distance between two arbitrary points
has a finite expectation.
The symbol $\mathbf{E}$ stands for the expectation in $\Omega$.
\begin{proposition}
For all $t>0$ we have
\begin{equation}\label{expe_fin}
\mathbf{E}(m_0((0,0),(t,0)))<+\infty
\end{equation}
Furthermore,  the limit
$$
\tau_0=\lim\limits_{t\to+\infty} \frac{m_0((0,0),(t,0))}t
$$
exists almost surely and is deterministic.
\end{proposition}
%
%
\begin{proof}
 We say that a set $S\subset\mathbb Z^2$ is $l^\infty$-connected  if for any two points $i$ and $j$ in $S$  there is a path
 $i=i_0,i_1,\ldots, i_m=j$ in $S$ such that $|i_k-i_{k-1}|_\infty=1$, $k=1,\ldots,m$.

 Consider all $l^\infty$-connected sets in $\mathbb Z^2$ of size $n$ that contain the origin.
 According to \cite[Proof of Theorem 4.20]{Grim99} for any $n\geq0$ the number of such sets is not greater than $C_2^n$ for some constant
 $C_2>0$.

 Next we choose $L$, $K$ and $\alpha>0$ in such a way that $C^2_2\gamma<\frac14$, where $\gamma$ is defined in
 \eqref{prob_cal_e}.

 We say that a site $j\in \mathbb Z^2$ is open if conditions $(\mathbf{c}_1)$-$(\mathbf{c}_3)$ in the proof of Lemma \ref{chalpha} are satisfied; otherwise $j$ is closed.
 The probability that a $l^\infty$-connected  set in $\mathbb Z^2$ consists of closed points, has size $n$ and is a maximum $l^\infty$-connected component of closed points does not exceed $\gamma^n$. We denote such a set by $S(n)$.

 Consider the sets
 $$
\mathcal{S}_0(n)= \bigcup\limits_{j\in S(n)}\big(Q_{5L}+10Lj\big),
$$
$$
\mathcal{S}_1(n)=\mathcal{S}_0(n)\ \bigcup \big\{x\in\mathbb R^2\,:\,\mathrm{dist}_\infty(x,10LS(n))\leq 10L\big\}.
$$
If  $\mathcal{S}_0(n)$ contains $k$ points of $\mathcal{N}$, then the length of the shortest path from $(0,0)$ to $(1,0)$
does not exceed $(k+8nK)^2$.  The probability that $\mathcal{S}_0(n)$ contains exactly $k$ points of $\mathcal{N}$ is equal to
$$\frac{(100L^2n)^k}{k!}\exp(-100L^2n).$$ Denote $L_0=100L^2$.

The probability that $S(n)$ is a maximum connected component of closed sites and that $\mathcal{S}_0(n)$
contains exactly $k$ points of $\mathcal{N}$ is not greater than
$$
p_{kn}=\big(\gamma^n\big)^\frac12 \Big(\frac{(L_0n)^k}{k!}\exp(-L_0n)\Big)^{\frac12}.
$$
Summing up over all connected sets in $\mathbb Z^2$ that contain the origin and over all $k$ from $0$ to $+\infty$,
we obtain that the expectation of the shortest path from $(0,0)$ to $(1,0)$ admits the following upper bound:
\begin{eqnarray*}
&&\mathbf{E}  (m_0((0,0),(1,0)))\leq \sum_{n=0}^\infty\sum_{k=0}^{\infty}C_2^np_{kn}(k+8nK)^2
\\
&\leq&\sum_{n,k=0}^\infty\exp\big( (\log(C_2)+\frac12\log(\gamma))n\big)\Big(\frac{(L_0n)^k}{k!}\exp(-L_0n)\Big)^{\frac12}(k+8nK)^2.
\end{eqnarray*}
Since $\frac{(L_0n)^k}{k!}\exp(-L_0n)<1$, using the Stirling formula and considering our choice of $\gamma$,
one concludes that the series converges. This yields the relation in \eqref{expe_fin} for $t\leq 1$. For larger $t$
we use the subadditive property of ${m_0((0,0),(t,0))}$. Namely, for any $s_1,\,t_1$ and $s_2,\,t_2$ we have
$$
m_0((0,0),(s_2,t_2))\leq m_0((0,0),(s_1,t_1))+m_0((s_1,t_1),(s_2,t_2)).
$$
This ensures the relation $\mathbf{E}(m_0((0,0),(t,0)))<+\infty$ for any $t>0$.

The second statement now follows from the Kingman subadditive ergodic theorem, see \cite{Kingman} or \cite{Krengel} for details.
 \end{proof}

\begin{proposition}[isotropy  and uniformity of the surface tension]\label{unif}
We have
\begin{equation}\label{uniform_b}
\tau_0 =\lim_{t\to+\infty} {m_0(x,x+tv)\over t}
\end{equation}
for all $v\in S^1$, and the limit is uniform for $x=x(t)$ if $|x|\le Ct$ and $v\in S^1$.
\end{proposition}

\begin{proof}
Our first goal is to show that there exists a constant $C_0$ such that a.s.~for any $\varkappa>0$ and $c_1>0$  and for all $t\geq t_0(\omega,c_1)$
we have
\begin{equation}\label{loca_est}
m_0(x,y)\leq C_0|x-y|+\varkappa t+\sqrt{t}
\end{equation}
for all $x$ and $y$ from the cube $\{x\in\mathbb R^2\,:\,|x|_\infty\leq c_1t\}$.   To this end we use again the definition of a cube $Q_{5L}$ given in the proof of Lemma \ref{chalpha} and recall that a site $j\in\mathbb Z^2$ is open if conditions
$(\mathbf{c}_1)$--$(\mathbf{c}_3)$ are fulfilled. We then choose the parameter $\gamma$ in \eqref{prob_cal_e}
sufficiently small
so that the open sites form a.s.~an infinite open cluster that we call $\mathcal{C}$.
Then a.s.~for sufficiently large $t$ the diameter of any  $l^\infty$-connected component of sites in
in the complement to the infinite open cluster in  $\{x\in\mathbb R^2\,:\, |x|_\infty\leq (10L)^{-1}c_1t\}$ does not exceed $c_2\log t$ with $c_2>0$, see \cite{Grim99}.
Computing the probability to have in a cube of size  $c_2\log t$ more than $\sqrt{t}$ points of $\mathcal{N}$,
considering the fact that the number of such cubes centred at $j\in\mathbb Z^2$ and belonging to
$\{x\in\mathbb R^2\,:\, |x|_\infty\leq (10L)^{-1}c_1t\}$ grows polynomially in $t$ and using the Borel-Cantelli lemma
we conclude that a.s.~for sufficiently large $t$ we have
\begin{equation}\label{perv_est}
m_0(x,\tilde\pi_\alpha(x))\leq \sqrt{t},\qquad m_0(y,\tilde\pi_\alpha(y))\leq \sqrt{t},
\end{equation}
where $\tilde\pi_\alpha(x)$  is the nearest to $x$ vertex of the union of the Voronoi cells that contain points of the scaled
infinite open cluster   $10L\mathcal{C}$.  

From the results in \cite{GarMar} it follows that a.s.~for sufficiently large $t$, for any two points $j^1$ and $j^2$ of the open
infinite cluster such that $j^1,\,j^2\in\{x\in\mathbb R^2\,:\, |x|_\infty\leq (10L)^{-1}c_1t\}$, and for any $\varkappa>0$
the cluster distance between $j^1$ and $j^2$ is not greater than $C_2|j^1-j^2|+\varkappa t$; here $C_2$ is a positive constant that does not depend on $\varkappa$. 
 Combining this estimate with \eqref{perv_est} we obtain \eqref{loca_est}.

Next, we are going to show that for any $x\in\mathbb R^2$ with $|x|\leq C$ and any $v\in S^1$ the limit relation
\begin{equation}\label{indiv_ineq}
\tau_0 =\lim_{t\to+\infty} {m_0(tx,tx+tv)\over t}
\end{equation}
holds a.s. In view of \eqref{loca_est} it suffices to prove this relation for integer $t$ that tends to $\infty$. In the remaining part of the proof we call this parameter $n$ instead of $t$.

We fix a small positive $\theta>0$ and denote by $\mathcal{A}_N$ the event
$$
\mathcal{A}_N=\big\{\omega\in\Omega\,:\, \Big|\frac{m_0(0,kv)}{k}-\tau_0\Big|\leq\theta\ \hbox{for all }k\geq N\big\}.
$$
Since $\mathbf{P}(\mathcal{A}_N)$ tends to $1$ as $N\to\infty$, for any $\delta>0$ there exists $N_0=N_0(\delta)$ such that
$$
\mathbf{P}(\mathcal{A}_{N_0})\geq 1-\delta.
$$
By the Birkhoff ergodic theorem a.s.~for any $\nu>0$ and $\varkappa>0$ there exists $k_0=k_0(\omega, \nu,\varkappa)$
such that
$$
\Big|\frac1k\sum\limits_{j=1}^k\mathbf{1}_{\mathcal{A}_{N_0}}(T_{jx}\omega)-\mathbf{P}(\mathcal{A}_{N_0})\Big|\leq \nu
$$
for all $k\geq \frac12 k_0$ and moreover inequality \eqref{loca_est} holds for all such $k$.  We assume that $\nu$ and $\delta$
are small enough so that $3(\nu+\delta)\leq\frac12$.

For $k\geq k_0$ denote by $\ell$ the maximum of integers $j$ such that $j>k+1$ and for all $i\in(k,j)$ we have
$T_{ix}\omega\not\in \mathcal{A}_{N_0}$.

Let $M$ be the number of unities in the sequence $\big\{\mathbf{1}_{\mathcal{A}_{N_0}}(T_{ix}\omega)\big\}_{i=1}^k$.
By the definition of $\ell$, the number of unities in  
$\big\{\mathbf{1}_{\mathcal{A}_{N_0}}(T_{ix}\omega)\big\}_{i=1}^{k+\ell}$ is equal to $M$ as well. 

Since $k+\ell>k_0$, we have
$$
\nu>\Big|\frac M{k+\ell}-\mathbf{P}(\mathcal{A}_{N_0})\Big|=
\Big|1-\mathbf{P}(\mathcal{A}_{N_0})-\frac{\ell+(k-M)}{k+\ell}\Big|.
$$
This yields
$$
\frac{\ell+(k-M)}{k+\ell}<\nu+1-\mathbf{P}(\mathcal{A}_{N_0})\leq\nu+\delta.
$$
Since $k-M\geq 0$, recalling that $\nu+\delta\leq \frac16$ we obtain $\ell\leq 2(\nu+\delta)k$.

For an arbitrary $k>\max\big(k_0, N_0\big)$ and $L=3(\nu+\delta)k$ there exists 
$n\in [k,k+L]$ such that $T_{nx}\omega\in\mathcal{A}_{N_0}$. Then 
we have
\begin{equation}\label{alm_req}
\Big|\frac1k m^\omega_0(nx,nx+kv)-\tau_0\Big|=\Big|\frac1k m^{T_{n\!x}\omega}_0(0,kv)-\tau_0\Big|\leq\theta.
\end{equation}
Since $n-k\leq 3(\nu+\delta)k$ and $k>k_0$, then by \eqref{loca_est}
$$
\big|m_0(nx,nx+kv)-m_0(kx,kx+kv)\big|\leq [3C_0C(\nu+\delta)+\varkappa]k+\sqrt{k}.
$$
Dividing by $k$ and considering \eqref{alm_req} we obtain
$$
\Big|\frac1k m^\omega_0(kx,kx+kv)-\tau_0\Big|\leq \theta+ [3C_0C(\nu+\delta)+\varkappa]+\frac1{\sqrt{k}}.
$$
It remains to take into account the fact that $\theta$, $\nu$, $\delta$ and $\varkappa$ are arbitrary positive number,
and \eqref{indiv_ineq} follows.

In view of estimate \eqref{loca_est} the pointwise convergence in \eqref{indiv_ineq} implies the uniform
convergence in \eqref{uniform_b} for $|x|\leq Ct$.  This completes the proof.
\end{proof}

\begin{proposition}[coerciveness of the surface tension]
We have $\tau_0>0$.
\end{proposition}

\begin{proof} Given $t>0$ take a minimal path $\{e_k\}$ for $m_0((0,0),(0,t))$.
We can apply Lemma \ref{pimentel} with $\e=1/t$, $R=1$, $\gamma=1/2$, and
$P\in\Pi$ with $e_k\subset P$ for all $k$.
We then have
$$
t(1+o(1))\le \#{\bf A}(P)\le C\#\{C_i: C_i\subset P\}\le C\#\{e_k\},
$$
which shows the claim, since the constant in this estimate are independent of $t$.
\end{proof}

\begin{proposition}\label{est-dist}
 There exists a constant $C_0$ such that if $t$ is large enough then if  $\{e_k\}$ is a test path for $m_0(x,x+tv)$ with $x$ as in Proposition {\rm\ref{unif}} with $\#(\{e_k\})\le t M$, then each point of $\{e_k\}$ is at most at distance $C_0Mt$ from $x$.
\end{proposition}

\begin{proof}
It suffices to apply Lemma \ref{pimentel} to the set of all Voronoi cells with non empty intersection with $\bigcup_ke_k$ and $\e=1/t$. We then cover $\bigcup_ke_k$ with the union of at most $2C_0 Mt$ cubes, from which the claim follows.
\end{proof}

\subsection{Computation of the $\Gamma$-limit}

{\em Lower bound.} We use an argument typical of the blow-up technique \cite{FM,BMS}.

Let $u^\e\to A$. Since $A$ is of finite perimeter, with fixed $\sigma>0$ and $\delta>0$ we consider a disjoint finite family of rectangles
$$
R_i=\Bigl\{x: |\langle x-x_i,\nu_i\rangle|\le \delta\rho_i, |\langle x-x_i,\nu_i^\perp\rangle|\le {1\over 2}\rho_i\Bigr\}
$$
such that
$$
{\cal H}^1\Bigl(\partial A\setminus \bigcup_i R_i\Bigr)\le\sigma
\qquad\hbox{
and
}\qquad
\Bigl|\sum_i\rho_i-{\cal H}^1(\partial A)\Bigr|\le\sigma.
$$

Since $A_\e\to A$ we may assume that
$$
{\cal L}^2(A_\e\cap R^+_i) = o(1), \qquad {\cal L}^2((A\setminus A_\e)\cap R^+_i) = o(1)
$$
as $\e\to 0$, where
$$
R^\pm_i= R_i\pm 2\delta\rho_i\nu_i.
$$

We now fix an index $i$. We use the channel property in Lemma \ref{chalpha} to find a path $\{\e C^+_j\}$
joining the two sides of $R^+_i$ parallel to $\nu_i$, with $j$ endpoints of segments of a path in $\E_\alpha$, and such that
$$
{\cal L}^2\Bigl(A_\e\cap R^+_i\cap \e \bigcup_j C^+_j\Bigr)\le {\e\over  C_\delta\rho_i}
{\cal L}^2(A_\e\cap R^+_i),
$$
which follows from the existence of a number of disjoints paths proportional to $\rho_i$.

Note that, since $|\e C^+_j|\ge \pi\e^2\alpha^2$, we have
$$
\#\{ j: \e C^+_j\subset A_\e\}\le {1\over \pi\e^2\alpha^2} {\cal L}^2\Bigl(A_\e\cap R^+_i\cap \e \bigcup_j C^+_j\Bigr)
\le  {1\over \e\pi C_\delta\rho_i\alpha^2}
{\cal L}^2(A_\e\cap R^+_i).
$$

Similarly, we define $\{C^-_j\}$ joining the two sides of $R^-_i$ parallel to $\nu_i$, and such that
$$
{\cal L}^2\Bigl((A\setminus A_\e)\cap R^-_i\cap \e \bigcup_j C^-_j\Bigr)\le {\e\over C_\delta\rho_i}
{\cal L}^2((A\setminus A_\e)\cap R^-_i),
$$
so that
$$
\#\{ j: \e C^-_j\subset (A\setminus A_\e)\}
\le  {1\over \e \pi C_\delta\rho_i\alpha^2}
{\cal L}^2(((A\setminus A_\e)\cap R^-_i).
$$

We define $U^+_\e$ as the connected component of $R^+_i\setminus \e\bigcup_j C^+_j$ containing the upper side
$S^+_i=\{x\in R^+_i: \langle x-x_i,\nu_i\rangle= 3 \delta\rho_i\}$ and $U^-_\e$ as the connected component of $R^-_i\setminus \e\bigcup_j C^-_j$ containing the lower side
$S^-_i=\{x\in R^-_i: \langle x-x_i,\nu_i\rangle= -3 \delta\rho_i\}$, and define
$$
\widetilde A_\e=(A_\e\setminus U^+_\e)\cup U^-_\e.
$$

We now consider the connected component of the set $(R_i\cup R^+_i\cup R^-_i)\setminus \widetilde A_\e$ containing the upper side $S^+_i$. Note that this connected component does not contain $S^-_i$, so that it contains a path of edges $\{e^\e_k\}$ in $\cal V$
connecting the two sides of $R_i\cup R^+_i\cup R^-_i$ parallel to $\nu_i$. We denote by $x^\pm_{\e}$ the extreme points of this path.

Using Proposition \ref{unif}, we can now estimate
\begin{eqnarray*}
\#\{\hbox{edges of } \partial V_\e(u^\e) \hbox{ inside } R_i\}
&\ge&
\#\{\hbox{edges of } \partial A_\e \hbox{ inside } R_i\}
\\
&\ge& \#\{e^\e_k\}-  {1\over \e \pi C_\delta\rho_i\alpha^2}
 o(1)
\\
&\ge &
m_0(x^-_{\e},x^+_{\e})- {1\over \e \pi C_\delta\rho_i\alpha^2}
 o(1)\\
&\ge &
(\tau_0+o(1)){\rho_i\over\e}-  {1\over \e \pi C_\delta\rho_i\alpha^2}
 o(1).
\end{eqnarray*}
Summing up in $i$ we then get
$$
\liminf_{\e\to0}E_\e(u^\e)\ge \sum_i\rho_i \tau_0\ge \tau_0({\cal H}^1(\partial A)-\sigma)
$$
and prove the claim by the arbitrariness of  $\sigma$.

\bigskip
{\it Upper bound.} By an approximation argument \cite{GCB,LN98} it is sufficient to prove the upper bound for polyhedral sets. Moreover, we can just deal with a single connected bounded polyhedron $A$ with a connected boundary since all other cases can be reduced to that by considering union or complements of such sets.

We write the boundary of $A$ as the union of segments $[x_{j-1}, x_j]$ with endpoints $x_0,\ldots, x_N\in\rr^2$ with $x_N=x_0$. With fixed
$m\in\NN$ and $\delta>0$, for all $j\in\{1,\ldots,N\}$ and  $l\in \{1,\ldots, m\}$
we consider a non-intersecting path $\{e^{j,l}_k\}$ in $\V$ between $\pi_0(x^\e_{j,m-1})$ and $\pi_0(x^\e_{j,m})$, where
$$
x^\e_{j,m}={1\over\e}\Bigl( x_{j-1}+{l\over m}(x_j-x_{j-1})\Bigr),
$$
such that
\begin{equation}
\#\{e^{j,l}_k\}\le {1\over m\e}|x_j-x_{j-1}|(\tau_0 +\delta).
\end{equation}

Denoting the union of the rescaled paths
$$
B^{\delta, m}_\e=\e\bigcup_{j,l,k}e^{j,l}_k
$$
let $A^{\delta, m}_\e$ be the complement of the infinite connected component of $\rr^2\setminus B^{\delta, m}_\e$ (note that the paths $\{e^{j,l}_k\}$ may intersect, so that there may be more than one bounded connected component of the complement of their union).
If $u^\e$ is defined as
\begin{equation}
u^\e_i=\begin{cases} 1 & \hbox{ if }i\in A^{\delta, m}_\e\\
 -1 & \hbox{ if }i\not\in A^{\delta, m}_\e,\end{cases}
\end{equation}
then we have
\begin{eqnarray}\label{estiam}\nonumber
E_\e(u^\e)&\le& \e\sum_{j,l,k}\#\{e^{j,l}_k\}\le \sum_{j,l}{1\over m}|x_j-x_{j-1}|(\tau_0 +\delta) \\
&=& \HH^1(\partial A)(\tau_0+\delta),
\end{eqnarray}
since the boundary of $A^{\delta, m}_\e$ is contained in $B^{\delta, m}_\e$.

By Lemma \ref{set-compactness}, thanks to \eqref{estiam} these sets converge as $\e\to 0$ to a set of finite perimeter $A^{\delta, m}$, and
\begin{equation}
\Gamma\hbox{-}\limsup_{\e\to0} E_\e(A^{\delta, m})\le \HH^1(\partial A)(\tau_0+\delta).
\end{equation}

Thanks to Proposition \ref{est-dist}
%
each point of $\e \{e^{j,l}_k\}$ is at most at a distance $C/m$ from the segment
$[\e \pi_0(x^\e_{j,m-1}),\e\pi_0(x^\e_{j,m})]$, and hence, since  $$\lim_{\e\to0}\e \pi_0(x^\e_{j,m})=x_{j-1}+{l\over m}(x_j-x_{j-1}),$$
the boundary of $A^{\delta, m}$ is contained in a $C/m$-neighbourhood of $\partial A$. This implies that $A^{\delta, m}$ converge to $A$ as $m\to+\infty$ independently of $\delta$. By the lower semicontinuity of the $\Gamma$-limsup \cite{GCB} we then deduce that
\begin{eqnarray*}
\Gamma\hbox{-}\limsup_{\e\to0} E_\e(A)\le\lim_{m\to+\infty}
\Gamma\hbox{-}\limsup_{\e\to0} E_\e(A^{\delta, m})\le \HH^1(\partial A)(\tau_0+\delta),
\end{eqnarray*}
and the claim is proved.

\section{Approximate surface tensions}
In this final section we consider the restriction of the energies $E_\e$ to (spin functions with related) sets whose boundary is composed of edges of $\alpha$-regular Voronoi cells. We denote by $E^\alpha_\e$ such energies. Note that in this case $E^\alpha_\e(u^\e)$ immediately gives the equiboundedness of the perimeter of the sets $V_\e(u^\e)$ and hence their precompactness. We briefly describe the limit of $E^\alpha_\e$ at fixed $\alpha$.

With given $\alpha<\alpha_0$ as in Proposition \ref{conn-alfa} we define for all $x\in\rr^2$
$$
\pi_\alpha(x)=\hbox{ closest point of $\N_\alpha^*$ to $x$}.
$$
For almost all $x$ this point is uniquely defined. For the remaining points we choose one of the closest points of $\N_\alpha^*$ to $x$.
For all $x,y\in \rr^2$ we set
$$
m_\alpha(x,y)=\min\{\#\{e_k\}: \{e_k\} \hbox{ is a path in $\V_\alpha$ connecting $\pi_\alpha(x)$ and $\pi_\alpha(y)$}\}.
$$

%


%

\begin{proposition}\label{sute-alfa}
For all $\alpha<\alpha_0$ a.s.~the limit
$$
\tau_\alpha =\lim_{t\to+\infty} {m_\alpha(x,x+tv)\over t}
$$
exists for all $v\in S^1$, and the limit is uniform for $x=x(t)$ if $|x|\le Ct$ and $v\in S^1$.
Furthermore $\tau_\alpha\in (0,+\infty)$.
\end{proposition}
\begin{proof}
The proof follows that for $\tau_0$, and is actually simpler as bounds for $m_\alpha(x,x+tv)$ are easier
\end{proof}

\begin{theorem}[homogenization on the $\alpha$-cluster]
For $\alpha<\alpha_0$ almost surely there exists the $\Gamma$-limit of $E^\alpha_\e$ and it equals $\tau_\alpha{\cal H}^1(\partial A)$.
\end{theorem}

\begin{proof}
The proof is the same as for the homogenization theorem in the previous section, taking care of using the same $\alpha$ as the one labeling the energies in the proof of the lower bound. Conversely, for the proof of the upper inequality, it is not necessary to use Proposition \ref{est-dist}.
\end{proof}

\begin{proposition}\label{sute}
We have
$
\inf\limits_{\alpha<\alpha_0}\tau_\alpha
=\lim\limits_{\alpha\to0}\tau_\alpha$.
\end{proposition}
\begin{proof}
Choose $\alpha_0>0$ in such a way that for some $L$ and $K$ we have
$$
\mathbf{P}(\mathcal{E}(L,K, \alpha_0,j))>p_{\rm cr}.
$$
 It suffices to show that $\tau_{\alpha_1}\leq \tau_{\alpha_2}$,
if $0<\alpha_1<\alpha_2\leq \alpha_0$.  Since $\mathcal{N}^*_{\alpha_2}\subset \mathcal{N}^*_{\alpha_1}$, then
\begin{equation}\label{mixed_ineq}
\begin{array}{c}
\displaystyle
\min\big\{\#\{e_k\}: \{e_k\} \hbox{ is a path in $\V_{\alpha_1}$ connecting $\pi_{\alpha_2}(x_0)$ and $\pi_{\alpha_2}(x_t)$}\big\}\\
\displaystyle
\leq m_{\alpha_2}(x_0,x_t),
\end{array}
\end{equation}
where  $x_0=0$ and $x_t=(t,0)$.
We should estimate
$$
\min\big\{\#\{e_k\}: \{e_k\} \hbox{ is a path in $\V_{\alpha_1}$ connecting $\pi_{\alpha_1}(x_t)$ and $\pi_{\alpha_2}(x_t)$}\big\}.
$$
To this end we consider the cubes $Q_{5L}+10Lj$, $j\in\mathbb Z^2$, that were introduced in the proof of Lemma \ref{chalpha}
and take those of them that satisfy conditions $(\mathbf{c}_1)$--$(\mathbf{c}_3)$ for $\alpha=\alpha_0$.
 Under our choice of  $\alpha_0$ a.s.~these exists a unique infinite cluster of such cubes. The complement to the  infinite cluster
 consists of connected bounded sets. Moreover, according to \cite{Grim99}, for sufficiently large $t$ the maximal size of the connected
 components in the   complement to the  infinite cluster that have a non-trivial intersection with $[-2t,2t]^2]$ does not exceed
 $c\log(t)$.  This implies that the size of the maximal connected component of $[-2t,2t]^2]\setminus \mathcal{V}_{\alpha_0}$
 does not exceed $c\log(t)$.
  Since $\mathcal{N}^*_{\alpha_0}\subset\mathcal{N}^*_{\alpha_2}\subset\mathcal{N}^*_{\alpha_1}$, then $\pi_{\alpha_1}(x_t)$ and $\pi_{\alpha_2}(x_t)$ belong to the closure of the same connected component of $[-2t,2t]^2]\setminus \mathcal{V}_{\alpha_0}$.
  Therefore,
 $$
\lim\limits_{t\to\infty} \frac1t\min\big\{\#\{e_k\}: \{e_k\} \hbox{ is a path in $\V_{\alpha_1}$ connecting $\pi_{\alpha_1}(x_t)$ and $\pi_{\alpha_2}(x_t)$}\big\}=0.
 $$
 Similarly,
 $$
\lim\limits_{t\to\infty} \frac1t\min\big\{\#\{e_k\}: \{e_k\} \hbox{ is a path in $\V_{\alpha_1}$ connecting $\pi_{\alpha_1}(x_0)$ and $\pi_{\alpha_2}(x_0)$}\big\}=0.
 $$
 Combining these two relations with \eqref{mixed_ineq} we obtain the required inequality $\tau_{\alpha_1}\leq \tau_{\alpha_2}$.
\end{proof}

\subsection*{Acknowledgments.}
The authors acknowledge the MIUR Excellence Department Project awarded to the Department of Mathematics, University of Rome Tor Vergata, CUP E83C18000100006.


\begin{thebibliography}{26}\frenchspacing
\baselineskip12pt


\bibitem{ABC} R. Alicandro, A. Braides, and M. Cicalese,
Phase and anti-phase boundaries in binary discrete systems: a variational viewpoint.
{\it Netw. Heterog. Media}  {\bf1} (2006), 85--107

\bibitem{ACG} R. Alicandro, M. Cicalese, and A. Gloria.
 Integral representation results for energies defined on stochastic lattices and application to nonlinear elasticity.
{\em Arch. Ration. Mech. Anal.} {\bf 200} (2011), 881--943.

\bibitem{ACR}
R. Alicandro, M. Cicalese, and M. Ruf.
 Domain formation in magnetic polymer composites: an approach via stochastic homogenization.
{\em Arch. Ration. Mech. Anal.} {\bf 218} (2015),  945--984.

\bibitem{A} L. Ambrosio.
Existence theory for a new class of variational problems.
{\em Arch. Ration. Mech. Anal.} {\bf 111} (1990), 291--322

\bibitem{AB} {\rm L. Ambrosio and A. Braides}, Functionals
defined on partitions of sets of finite perimeter, II:
semicontinuity, relaxation and homogenization, {\it J. Math.
Pures. Appl.} {\bf 69} (1990), 307-333.

\bibitem{BBC}
A. Bach,  A. Braides, and M. Cicalese.
Discrete-to-continuum limits of multi-body systems with bulk and surface long-range interactions.
Preprint 2019.

\bibitem{BBZ}
A. Bach, A. Braides and C. I. Zeppieri.
Quantitative analysis of finite-difference approximations of free-discontinuity problems.
Preprint 2018.

\bibitem{BLB}X. Blanc, C. Le Bris,  and P.L. Lions.
The energy of some microscopic stochastic lattices.
{\em Arch. Ration. Mech. Anal.} {\bf184} (2007), 303--339.

\bibitem{BIV1} T. Bodineau, D. Ioffe, Y. Velenik,
Rigorous probabilistic analysis of equilibrium crystal shapes,
{\em  J. Math. Phys.} {\bf 41} (2000) 1033-1098.

\bibitem{LN98} {\rm A. Braides},
{\it Approximation of Free-Discontinuity Problems}, Lecture
Notes in Mathematics {\bf 1694}, Springer Verlag, Berlin, 1998.

\bibitem{GCB} {\rm A. Braides},
{\it $\Gamma$-convergence for Beginners}, Oxford University Press,
Oxford, 2002.

\bibitem{Handbook} {\rm A. Braides}, A handbook of $\Gamma$-convergence.
In {\it Handbook of Differential Equations.
Stationary Partial Differential Equations, Volume $3$}
(M. Chipot and P. Quittner, eds.), Elsevier, 2006.

\bibitem{BCR}
A. Braides, M. Cicalese, and M. Ruf.
 Continuum limit and stochastic homogenization of discrete ferromagnetic thin films.
 {\em Anal. PDE} {\bf11} (2018), 499--553.

\bibitem{BDV} A. Braides, A. Defranceschi and E. Vitali.
Homogenization of free discontinuity problems.
{\em Arch. Ration. Mech. Anal.} {\bf 135} (1996), 297--356.

\bibitem{BK}  A. Braides and L. Kreutz.
An integral-representation result for continuum limits of discrete energies with multi-body interactions.
{\it SIAM J. Math. Anal.} {\bf 50} (2018), 1485--1520

\bibitem{BMS} A. Braides, M. Maslennikov, and L. Sigalotti.
Homogenization by blow-up.
{\it Applicable Anal.} {\bf 87} (2008), 1341--1356.

\bibitem{BP-ARMA}
 A. Braides and A. Piatnitski,
Overall properties of a discrete membrane with randomly distributed defects.
{\em Arch. Ration. Mech. Anal.} {\bf 189} (2008), 301-323.

\bibitem{BP-JOSS}A. Braides and A. Piatnitski.
Variational problems with percolation: dilute spin systems at zero temperature
{\it J. Stat. Phys.} {\bf 149} (2012), 846--864

\bibitem{BP-JFA}
A. Braides and A. Piatnitski.
Homogenization of surface and length energies for spin systems.
{\it  J. Funct. Anal.} {\bf 264} (2013), 1296--1328.


\bibitem{CDL} L.A. Caffarelli and R. de la Llave. Planelike minimizers in periodic media.
{\it Comm. Pure Appl. Math.} {\bf 54} (2001), 1403--1441.

\bibitem{DZ}F. Cagnetti, G. Dal Maso, L. Scardia, and C. I. Zeppieri.
$\Gamma$-convergence of free-discontinuity problems.
{\em Ann. Inst. H. Poincar\'e Anal. Non Lin\'eaire} {\bf 36} (2019), 1035--1079.

\bibitem{Calka} P. Calka. Precise formulae for the distributions of the principal geometric characteristics of the typical cells of a two-dimensional Poisson-Voronoi tessellation and a Poisson line process. {\em Adv. Appl. Prob.}  {\bf 35} (2003), 551-562. 



%



\bibitem{DV_J} D. J. Daley and D. Vere-Jones, {\sl An introduction to the theory of point processes},
Springer-Verlag, New York, 1988.

\bibitem{FM} I. Fonseca and S. M\"uller.
Quasi-convex Integrands and lower semicontinuity in $L^1$.
{\em SIAM J. Math. Anal.} {\bf  23} (1992), 1081--1098.

\bibitem{GarMar} O. Garet  and R. Marchand,  Asymptotic shape
for the chemical distance and first passage
percolation on the infinite Bernoulli cluster. {\it ESAIM
Probab. Statist.} {\bf 8} (2004), 169-�199.

\bibitem{GarMar1} O. Garet  and R. Marchand,
Large deviations for the chemical distance in supercritical Bernoulli percolation.
{\em Annals of Probability}  {\bf 35} (2007), 833-866.

\bibitem{Grim99}G. Grimmet. {\em Percolation}. Springer, Berlin, 1999.

\bibitem{JKO}  V.V. Jikov;  S.M. Kozlov; O.A. Oleinik,
{\it Homogenization of differential operators and integral functionals.},
Springer-Verlag, Berlin, 1994.


\bibitem{Kes} H. Kesten,  {\it Percolation Theory for Mathematicians}.
Progress in Probability and Statistics, 2. Birkh\"auser, Boston,
1982


\bibitem{Kingman}
JFC. Kingman. Subadditive ergodic theory. The Annals of Probability, {\bf 1} (1973), 883--899.

\bibitem{Krengel}
U. Krengel. {\it Ergodic Theorems}. W. de Gruyter, Berlin, 2011.

\bibitem{AK} {U. Krengel and R. Pyke}, Uniform pointwise ergodic theorems for classes of averaging
sets and multiparameter subadditive processes. {\em Stoch. Proc. Appl. }  {\bf 26} (1987) 298--296.

\bibitem{Maggi}
F. Maggi. {\it Sets of Finite Perimeter and Geometric Variational Problems: an Introduction to Geometric Measure Theory.}
Cambridge University Press, Cambridge, 2012.


\bibitem{Pi} L.P.R. Pimentel.
On some fundamental aspects of polyominoes on random Voronoi tilings.
{\em Brazilian Journal of Probability and Statistics}
{\bf 27} (2013),  54--69


\end{thebibliography}
\end{document}